\documentclass[11 pt]{amsart}
\usepackage{amsmath,amsthm}
\usepackage{mathrsfs}
\usepackage{amsfonts}
\usepackage{graphicx}
\usepackage{amssymb}
\graphicspath{{/texfiles/}}
\usepackage[shortlabels]{enumitem}
\usepackage{array}
\usepackage{wrapfig}
\usepackage{multirow}
\usepackage{tabularx}
\usepackage{tikz,lipsum,lmodern}
\usepackage{tikz-cd}
\usepackage[most]{tcolorbox}
\usepackage[margin=1in,footskip=0.25in]{geometry}
\usepackage{hyperref}
\usepackage{mathtools}
\usepackage{MnSymbol}

\newcommand{\R}{\mathbb{R}}
\newcommand{\C}{\mathbb{C}}
\newcommand{\Ox}{\mathcal{O}}
\newcommand{\bP}{\mathbb{P}}
\newcommand{\Z}{\mathbb{Z}}

\newcommand{\E}{\mathcal{E}}

\newtheorem{thm}{Theorem}[section]

\newtheorem{Proposition}[thm]{Proposition}

\theoremstyle{definition}
\newtheorem{definition}[thm]{Definition}
\newtheorem{lemma}[thm]{Lemma}

\newtheorem{remark}[thm]{Remark}

\newtheorem{theorem}[thm]{Theorem}
\newtheorem{corollary}[thm]{Corollary}

\mathtoolsset{showonlyrefs}

\title{Real line subbundles of real bundles on curves}
\author{Daniel Santiago}
\date{}
\address{Department of Mathematics, Harvard University, Cambridge, MA 02138}
\email{dsantiagoalvarez@math.harvard.edu}

\begin{document}
\begin{abstract}
For a stable real bundle $E$ of rank $2$ and degree $1$ on a real genus $2$ curve, we describe the action of the real structure of the curve on the set of $4$ maximal line subbundles of degree $0$ of $E$. This describes the Galois action on the set of lines through a real point in the moduli space of such bundles, and is a real algebraic extension of classical work of Newstead. Our proof is an application of techniques of Atiyah from the 1950's. We prove also results on real line subbundles in higher genus using work of Lange-Narasimhan. 
\end{abstract}
\maketitle

\section{Introduction}

Let $\Sigma$ be a Riemann surface of genus $g$, or equivalently a smooth projective curve over $\C$. For each choice of nonnegative rank $n$ and integer degree $d$, one has a moduli space $M(n,d)$ of semistable bundles with the corresponding rank and degree over $\Sigma$. There is a determinant map
\[\mathrm{det}: M(n,d) \to \mathrm{Pic}^d(\Sigma)\]
with fiber over an isomorphism class $\Lambda \in \mathrm{Pic}^d(\Sigma)$ given by the moduli space $M_{\Lambda}(n,d)$ of semistable bundles with determinant $\Lambda$. The following problem has been well studied in the literature (e.g \cite{lange1983maximal}, \cite{oxbury2000varieties}).\\

\noindent \textbf{Problem}: For a stable bundle $E \in M_{\Lambda}(2,d)$ of rank $2$ and degree $d$, what is the largest degree of a line subbundle $L \subset E$? How many line subbundles of this maximal degree are there?\\

By stability, the maximal possible degree of a line subbundle of $E \in M_{\Lambda}(2,d)$ is strictly less than $\frac{d}{2}$.  We will be interested in bundles $E$ containing a line subbbundle of the maximal degree $\lfloor \frac{d}{2} \rfloor$ for odd $d$ and $\frac{d}{2}-1$ for even $d$. For $d$ odd, there are only finitely many maximal line subbundles of degree $\lfloor \frac{d}{2} \rfloor$ (see Lemma $4.7$ in \cite{pal2016moduli}). For $d$ even, the analogous statement (see \cite{lange1983maximal}) holds for a general bundle on a curve of genus $g>2$. On a genus $2$ curve, work of Newstead (\cite{newstead1968stable}) implies that any stable bundle $E \in M_{\Lambda}(2,1)$ contains a degree $0$ line subbundle, and generically there are four such subbundles.\newline 

This note is dedicated to the study of maximal line subbundles of real bundles on real algebraic curves. A real curve $(\Sigma,\tau)$ is a Riemann Surface $\Sigma$ equipped with an antiholomorphic involution $\tau: \Sigma \to \Sigma$. The involution $\tau$ induces involutions on the spaces $M(n,d)$ by $E \mapsto \tau^*(\overline{E})$. For $\Lambda \in M(1,d)$ a line bundle fixed by this involution, we obtain a restricted action on $M_{\Lambda}(2,d)$. When $E \in M(n,d)$ is stable, the condition $\tau^*(\overline{E}) \cong E$ is equivalent to the existence of a lift

\[
\begin{tikzcd}
E \arrow[d] \arrow{r}{\tau_E} & E \arrow[d]\\
\Sigma \arrow{r}{\tau} & \Sigma
\end{tikzcd}
\]

\noindent of the real structure $\tau$ such that $\tau_E$ is antiholomorphic on fibers and squares to $1$ or $-1$. We will often simply use the name real to refer to a bundle $E$ such that $E \cong \tau^*(\overline{E})$. The fixed point sets  $M(n,d)^{\tau}, M_{\Lambda}(n,d)^{\tau}$ are nonempty real Lagrangians in $M(n,d),M_{\Lambda}(n,d)$ by \cite{biswas2010moduli}. Real curves and bundles are classified topologically in terms of the fixed circles $\Sigma^{\tau}$ of $\tau$, and the restriction of $(E,\tau_E)$ to these fixed circles. Our main result is the following

\begin{theorem}
\label{Genus-2-Thm-Intro}
Let $(\Sigma,\tau)$ be a real curve of genus $2$ with nonempty set of real points $\Sigma^{\tau}$. Depending on the topological type of $\tau$ and $\Lambda \in \mathrm{Pic}^1(\Sigma)^{\tau}$, exactly one of the following three behaviors occurs

\begin{enumerate}
\item Any real bundle $E \in M_{\Lambda}(2,1)^{\tau}$ contains a real line subbundle of degree $0$. There are nonempty open subsets of the moduli space $M_{\Lambda}(2,1)^{\tau}$ with dense union of real bundles containing exactly $2$ and $4$ real line subbundles of degree $0$.
\item Any real bundle $E \in M_{\Lambda}(2,1)^{\tau}$ contains $4$ real line subbundles of degree $0$. 
\item There are nonempty open sets of $M_{\Lambda}(2,1)^{\tau}$ with dense union of real bundles containing exactly $0,2,4$ real line subbundles of degree $0$.
\end{enumerate}

\end{theorem}

A precise statement is given in section \ref{Section-Genus-2-Nonsingular}. Interestingly, the boundaries of the open sets occuring in Theorem \ref{Genus-2-Thm-Intro} consist of the wobbly bundles studied by Donagi, Drinfeld, and Laumon in the context of the geometric Langlands correspondence (see Remark \ref{Wobbly-Remark} and \cite{Pal2017},\cite{laumon1988analogue},\cite{donagi2008geometric}). Our proof uses classical work of Atiyah \cite{atiyah1955complex}, as exposed in modern language by Donaldson \cite{donaldson2021atiyah}. The key idea is to study extension classes of line bundles by means of their associated divisors in $\mathrm{Sym}^3(\Sigma)$.

Using work of Lange-Narasimhan (\cite{lange1983maximal}), we classify topological types in genus $2$ of maximal line subbundles of a real bundle in $M_{\Lambda}(2,1)^{\tau}$ in Theorem \ref{Topological-Types-Theorem}. We also prove the following description of real line subbundles in higher genus.

\label{Higher-Genus-Thm-Intro}
 \begin{theorem}
 Let $(\Sigma,\tau)$ be a real curve of genus $g\geq3$ and $\Lambda,\beta$ real line bundles of degree $1,0$ on $\Sigma$. The following hold 
 \begin{enumerate}
 \item A real bundle $E \in M_{\Lambda}(2,1)$ which contains a maximal line subbundle $n \subset F$ of degree $0$ in fact contains a real line subbundle of maximal degree. Generically, $E$ has the unique maximal line subbundle $n$ of degree $0$ and thus $n$ is real.
 \item For $g>3$, the same statement is true for line subbundles $n \subset F$ of degree $-1$ of a real bundle $F \in M_{\Lambda}(2,0)$.
\end{enumerate}
\end{theorem}






The line subbundles of a stable rank $2$ bundle $E$ of degree $0$ determine the ways it lies in a projective space of extensions, which in turn determine the rational curves of degree $1$ passing through $E$ in the moduli space $M_{\Lambda}(2,1)$ (see \cite{munoz1999quantum}). Theorem \ref{Genus-2-Thm-Intro} thus describes the action of the real structure on the set of four lines through a real point in the Fano threefold $M_{\Lambda}(2,1)$ (see Section \ref{Real-Extension-Lines-Section}). In particular, it implies that the real moduli space $M_{\Lambda}(2,1)^{\tau}$ is not in general covered by real rational lines in contrast to the complex moduli space. Our results on rational curves are potentially useful for determining the real genus $0$ Gromov-Witten invariants (\cite{farajzadeh2016counting}) and Lagrangian floer homology of the moduli spaces $M_{\Lambda}(2,1)^{\tau}$.

\subsection{Outline of paper} We give a summary of the topological classification of real bundles and curves in section \ref{Real-Curves-Section}. In section \ref{Real-Extension-Lines-Section}, we discuss real rational degree one curves in the moduli space $M_{\Lambda}(2,1)$. Section \ref{Section-Genus-2-Nonsingular} discusses the genus $2$ case, and section \ref{Lange-Narasimhan-Applications} covers the applications of Lange-Narasimhan's results. We prove theorem \ref{Genus-2-Thm-Intro} in sections \ref{Section-Genus-2-Nonsingular},\ref{Lange-Narasimhan-Applications} and theorem \ref{Higher-Genus-Thm-Intro} in section \ref{Lange-Narasimhan-Applications}.

\subsection{Acknowledgements} I am grateful for the support of my advisor Professor Peter Kronheimer, as well as Thomas John Baird and Ollie Thakar for helpful discussions. This project was supported by the Simons Collaboration on New Structures in Low Dimensional Topology.

\section{Real curves and real hyperelliptic curves}
\label{Real-Curves-Section}

We begin by recalling the topological classification of real curves $(\Sigma,\tau)$ of genus $g = g(\Sigma)$. Let $\Sigma^{\tau}$ be the set of fixed points of $\tau$, which is some disjoint union of circles. To the pair $(\Sigma,\tau)$, we associate a pair $(n,a)$, where $n$ is the number of components of $\Sigma^{\tau}$, and $a$ is $0$ or $1$  according to wether $\Sigma \setminus \Sigma^{\tau}$ is disconnected or connected respectively. The pairs classify $(\Sigma,\tau)$ up to homeomorphism as follows (See \cite{baraglia2014higgs},\cite{baird2016cohomology}).

\begin{Proposition}
\label{Top-Classification-Real-Curves}
\upshape 
The invariants $(n,a)$ associated to a real curve $(\Sigma,\tau)$ satisfy the following conditions.
\begin{center}
\begin{itemize}
\item $0 \leq n \leq g+1$
\item If $n=0$ then $a = 1$ and if $n = g+1$ then $a = 0$
\item If $a=0$ then $n \equiv g+1$ mod $2$.
\end{itemize}
\end{center}

Conversely any pair $(n,a)$ satisfying these conditions is associated to some real curve, unique up to homeomorphism.\noindent
\end{Proposition}

There is a corresponding classification of topological real and quaternionic bundles $(E,\tau_E)$ on $(\Sigma,\tau)$. The fixed point set $E^{\tau_E}$ of such a bundle is a real vector bundle on $\Sigma^{\tau}$ of the same rank as $E$, and is hence classified by its Stiefel Whitney class $w_1(E^{\tau_E}) \in H^1(\Sigma^{\tau},\Z_2)$.

\begin{Proposition}
\label{Top-Classification-Real-Bundles}
\upshape 
Topological real vector bundles $(E,\tau_E)$ with $\tau_E^2 =1$ over a real curve $(\Sigma,\tau)$ are classified by rank $r$, degree $d$ and Stiefel Whitney Class $w_1(E^{\tau_E})$ such that 
\[d \equiv w_1(E^{\tau_E})([\Sigma^{\tau}]) \text{   mod   } 2\] 

In particular if $n \geq 1$, then there are $2^{n-1}$ isomorphism classes of topological real bundles on $(\Sigma,\tau)$ of any rank and degree.
Topological quaternionic bundles $(E,\tau_E)$, i.e. with $\tau_E^2 = -1$, over a real curve $(\Sigma,\tau)$ are classified by rank $r$ and degree $d$ subject to the condition that 
\[d \equiv r(g-1) \text{   mod   } 2\]
and that $n = 0$ (i.e $\Sigma^{\tau} = \emptyset$) if $r$ is odd.
\end{Proposition}
In particular there are no quaternionic bundles of odd degree on a real curve with fixed points. 

\subsection{Real hyperelliptic curves}

We fix the notation $c: \bP^1 \to \bP^1$ for the complex conjugation $[z:w] \mapsto [\overline{z}:\overline{w}
]$.

\begin{definition} A \textit{real hyperelliptic curve} is a hyperelliptic curve $(\Sigma,\tau)$ with a real structure $\tau$ and hyperelliptic projection $\pi: \Sigma \to \bP^1$ such that $\pi \circ \tau = c \circ \pi$. Note this excludes real curves $(\Sigma,\tau)$ which are both real and hyperelliptic, but for which the real structure $\tau$ does not descend to complex conjugation on $\bP^1$.

\end{definition}

For a real hyperelliptic curve $(\Sigma,\tau)$, the involution $\tau$ preserves the set of Weierstrass points $W$, each of which is either real or in a complex conjugate pair. The Weierstrass points are precisely the branch points of the map $\pi: \Sigma \to \bP^1$. One has the following description (\cite{baird2020hyperelliptic}) of the topological types for $(\Sigma,\tau)$ which arise as real hyperelliptic curves

\begin{Proposition} \upshape 
Write the set of Weierstrass points $W$ as $W = W_0 \sqcup W_+ \sqcup W_{-}$ according to zero, positive, or negative imaginary part, so that $|W_+|= |W_{-}|$ and $|W_0| = 2m$. Let $n$ be the number of components of $\Sigma^{\tau}$ as in Proposition \ref{Top-Classification-Real-Curves}.

\begin{itemize}
\item If $m > 0$ then $n = m$.
\item If $m = 0$ then $n = 1$.
\item If $1 \leq n \leq g$, then $\Sigma \setminus \Sigma^{\tau}$ is connected. If $n = g+1$ or $n=0$, then $\Sigma \setminus \Sigma^{\tau}$ is disconnected.
\end{itemize}

\end{Proposition}

\subsection{ Genus 2}
\label{Genus-2-Figures}
One can draw pictures representing each of the five possible cases for $(\Sigma,\tau)$ in genus $2$, in which case the curve $\Sigma$ is always hyperelliptic. The real hyperelliptic cases correspond to $0,2,4$ or $6$ real Weierstrass points. There is an additional case with no fixed circles, which is not real hyperelliptic and instead descends to the antipodal map on $\bP^1$. Explicit examples of algebraic curves with each topological type are given in \cite{cirre2003moduli}.\\
\indent In the following figures, both the involution $\tau$ and the hyperelliptic involution $\iota$ are drawn. A bold solid line indicates a fixed circle of $\tau$, while a bold dashed line indicates that $\tau$ is a composition of a Dehn twist by $\pi$ with a reflection in a neighborhood of the dashed circle. On the dashed circles $\tau$ restricts to the antipodal map. The hyperelliptic involution $\iota$ is rotation by $\pi$ about the axis containing the Weierstrass points labelled $w_1,\dots,w_6$. In each case $\iota,\tau$ commute.\\
\begin{figure}
\label{genus-2-type-(1,0)}
\centering
\includegraphics[width=93mm,height = 88mm]{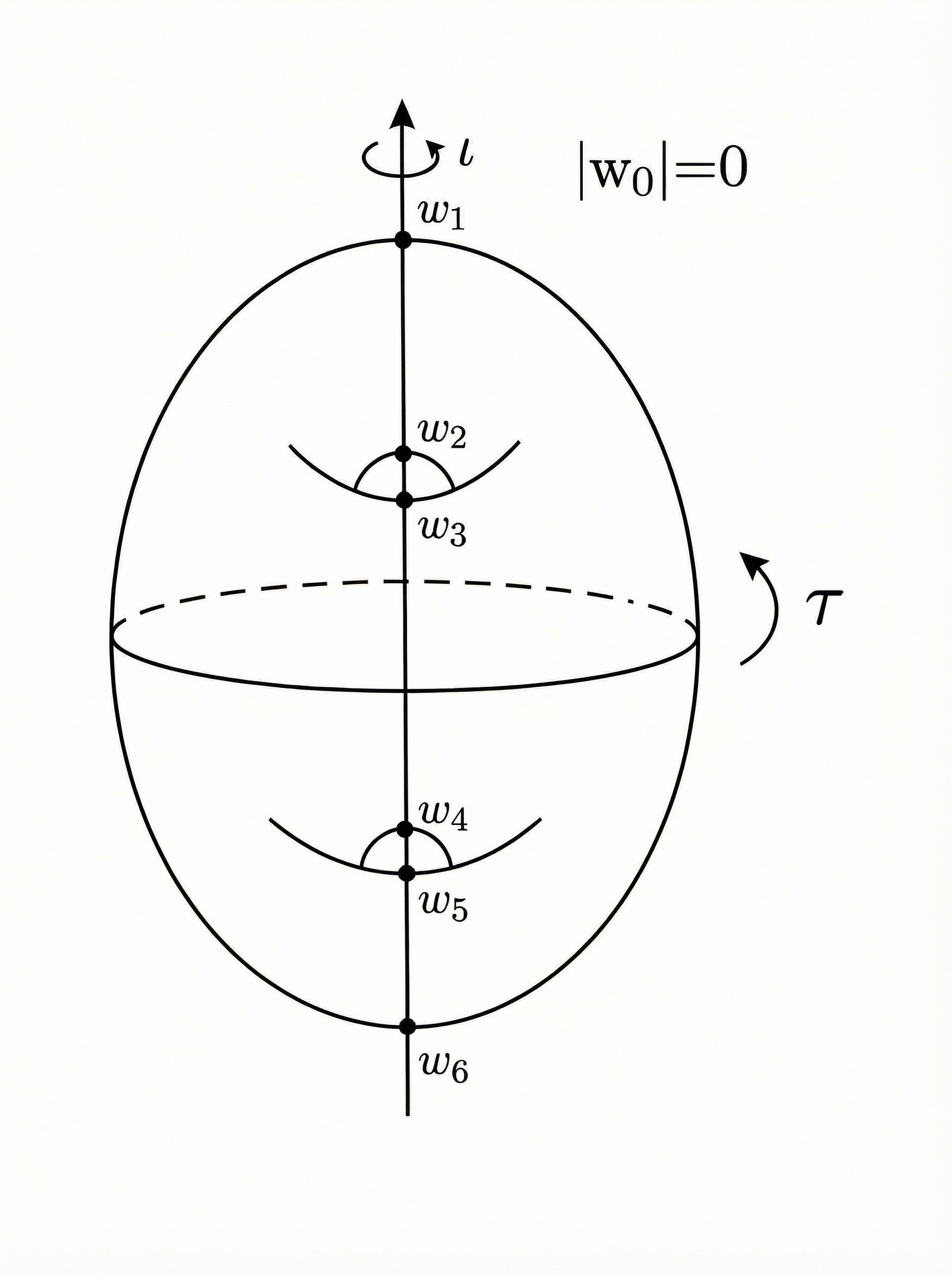}
\caption{A real hyperelliptic genus $2$ curve of type $(n,a) = (1,0)$ with one fixed circle and no real Weierstrass points.}
\end{figure}

\begin{figure}
\label{genus-2-type-(3,0)}
\centering
\includegraphics[width=90mm,height = 88mm]{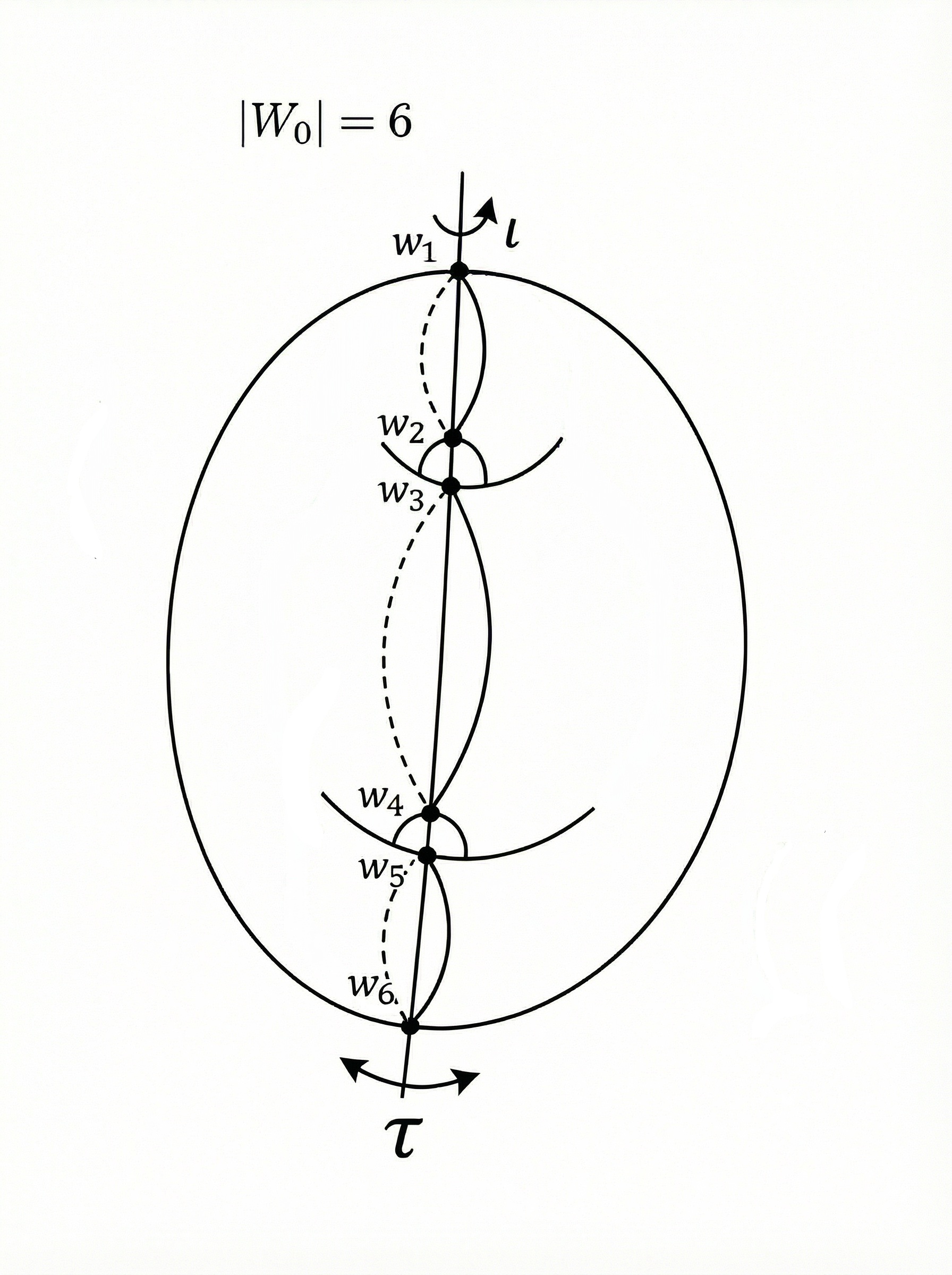}
\caption{A real hyperelliptic genus $2$ curve of type $(3,0)$ with three fixed circles and six real Weierstrass points.}
\end{figure}

\begin{figure}
\label{genus-2-types-(0,1),(1,1),(2,1)}
\centering
\includegraphics[width=175mm,height=95mm]{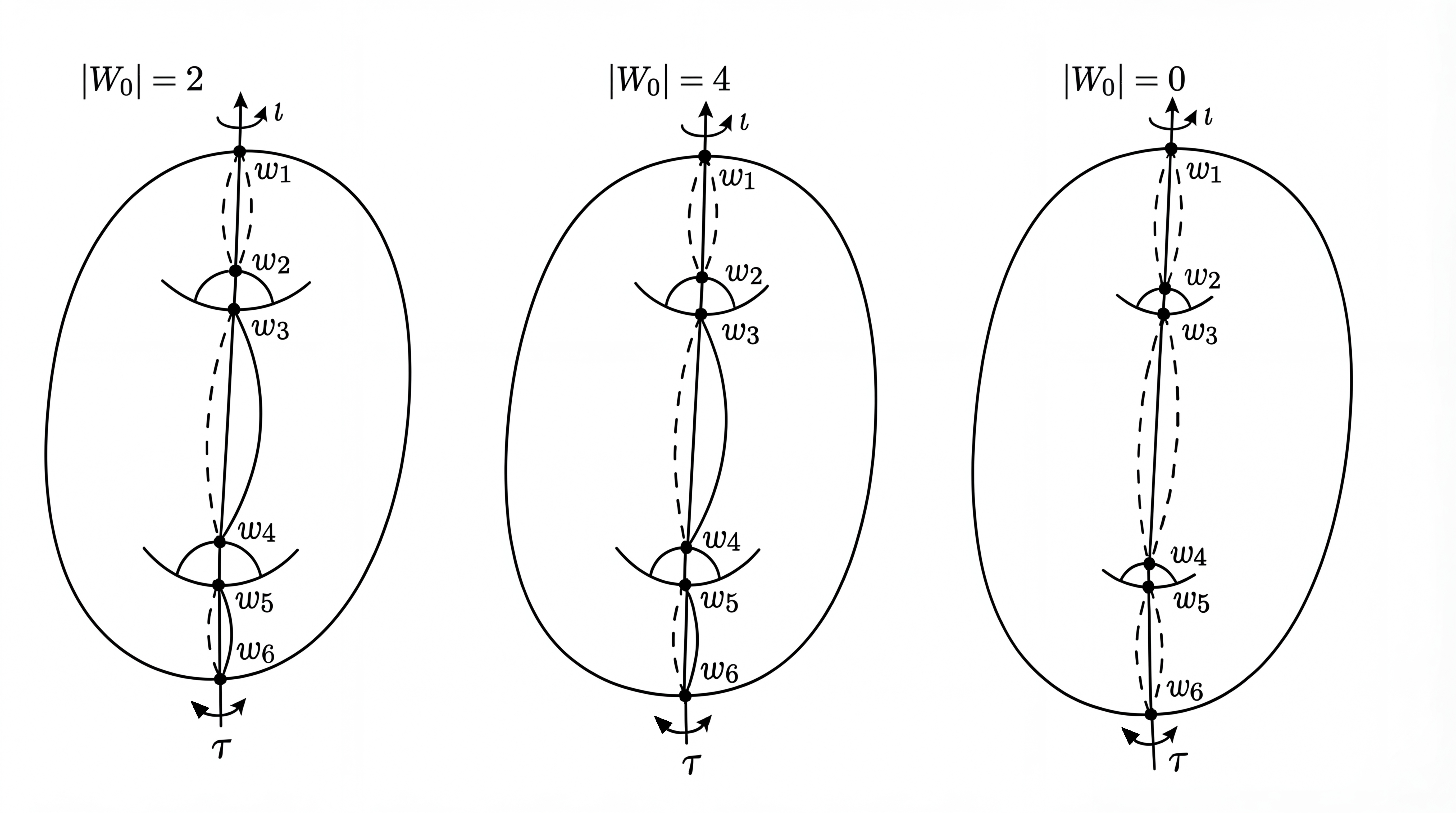}
\caption{From left to right, real genus $2$ curves of types $(1,1),(2,1),(0,1)$. These examples  have $1,2,0$ fixed circles respectively, and $2,4,0$ real Weierstrass points. Only the first two are real hyperelliptic. }
\end{figure}
\newpage

\subsection{Real holomorphic line bundles}

 Suppose now that $(\Sigma,\tau)$ is a real curve, with $\tau: \Sigma \to \Sigma$ an antiholomorphic involution. A real holomorphic line bundle $(L,\tau_L) \to (\Sigma,\tau)$ admits a $\tau_L$ invariant meromorphic section, which determines a divisor $\sum_i m_i p_i$ such that $\tau(\sum_i m_i p_i) = \sum_i m_ip_i$. We will refer to such divisors as \textit{real divisors}. We say that a circle $C \subset \Sigma^{\tau}$ is \textit{odd} for $L$ if the real line bundle $L^{\tau_L} \to C$ is a M\"obius bundle. One can characterize the odd circles (\cite{baird2020hyperelliptic}) by the property 
 \[\sum_{p_i \in C} m_i \equiv 1 \hspace{0.1 cm} \mathrm{ mod } \hspace{0.1 cm} 2\]

 i.e there are an odd number of points in $C$ counting multiplicity.

\section{Real Extension lines}
\label{Real-Extension-Lines-Section}

Let $(\Sigma,\tau)$ be a real curve of genus $g \geq 2$, and $(\Lambda,\tau_L) \to (\Sigma,\tau)$ a real line bundle of degree $1$. Following the work of Mu\~noz \cite{munoz1999quantum}, we will describe rational curves in the moduli space $M_{\Lambda}(2,1)$ of semistable bundles of rank $2$ with determinant $\Lambda$. The space $M_{\Lambda}(2,1)$ admits a symplectic structure which depends up to deformation only on the genus $g(\Sigma)$ of $\Sigma$. It holds that $\pi_2(M_{\Lambda}(2,1)) \cong \Z$. We denote by $A: S^2 \to M_{\Lambda}(2,1)$ the generator which has positive symplectic area. Let $M_{A}$ be the space of degree $1$ rational curves with respect to this generator

\[M_{A}= \{f: \bP^1 \to M_{\Lambda}(2,1) \text{ holomorphic }\mid f_*[\bP^1 ] = A\}\]

For each degree $d$, we denote by $J_d$ the component $\mathrm{Pic}^d(\Sigma)$ of the Picard group of $\Sigma$ consisting of line bundles of degree $d$. We denote by $\mathcal{L}_d \to \Sigma \times J_d$ a choice of universal line bundle on $\Sigma \times J_d$. Note that $\mathcal{L}_d$ it is not uniquely determined. Define for each $d$ the sheaf $\E_d \to J_d$ by 

\[\mathcal{E}xt_p^1(\Lambda \otimes \mathcal{L}_d^{-1},\mathcal{L}_d) = R^1p_*(\mathcal{L}_d^2 \otimes \Lambda^{-1})\]

where $p$ is the projection $p: \Sigma \times J_d \to \Sigma$. For $L \in J_d$, the fiber of $\E_d$ over $J$ is identified with the space $\mathrm{Ext}^1(\Lambda \otimes L^{-1},L)$ of extensions 

\[0 \to L \to E \to \Lambda \otimes L^{-1} \to 0\]

We focus now on the case $d=0$, in which case $\E_0$ is a bundle of rank $g$. Any nontrivial extension class as above determines a stable bundle $E$, and scaling does not change the stable isomorphism class, thus one obtains a map
\[i_{\Lambda}: \bP(\E_0) \to M_{\Lambda}(2,1)\]
where $\bP(\E_0)$ is the projectivized bundle with fibers $\bP_L := \bP(\mathrm{Ext}^1(\Lambda \otimes L^{-1},L))$. The bundle $\bP(\E_0)$ admits a natural symplectic structure, and we denote by $B$ the positive generator of $\pi_2(\bP(\E_0)) \cong \pi_2(\bP^{g-1}) \cong \Z$. One has the following fact

\begin{lemma} (\cite{munoz1999quantum}) 
\label{Degree-1-Lines} \upshape 
Let $M_B$ be the space

\[M_{B}= \{f: \bP^1 \to \bP(\E_0) \text{ holomorphic }\mid f_*[\bP^1 ] = B\}\]
Composition with $i_{\Lambda}$ induces an isomorphism $M_{B} \to M_{A}$. In particular, any element $f \in M_{A}$ factors uniquely through a map $ \bP^1 \to \bP(\E_0)$ landing in a fiber of $\bP(\E_0) \to J_0$. 
\end{lemma}

Since $\Lambda$ is real, the real structure $\tau$ induces a real structure on $M_{\Lambda}(2,1)$ and $J_0$ by conjugate pullback $E \mapsto \tau^*(\overline{E})$. We consider the action of the real structure on the projectivized bundle $\bP(\E_0)$.

\begin{Proposition} Suppose $\tau: \Sigma \to \Sigma$ fixes $n \geq 1$ circles. The fixed point set $J^{\tau}_0$ is isomorphic to $J^{\R}_{0} \times (\Z_2)^d$, where $J^{\R}_{0}$ is a real torus of dimension $g$ and $d = 2^{n-1}$.

\end{Proposition}

\begin{proof}
That $J^{\tau}$ has the form $J^{\R}_{\tau} \times (\Z_2)^d$ is proven in (\cite{baraglia2014higgs} Lemma 27, \cite{gross1981real}). The number of components $d$ is $2^{n-1}$ by Proposition \ref{Top-Classification-Real-Bundles}.

\end{proof}

\noindent Applying the involution $E \mapsto \tau^*(\overline{E})$ sends an extension

\[0 \to L \to E \to \Lambda \otimes L^{-1} \to 0 \]
to an extension \[0 \to \tau^*(\overline{L}) \to \tau^*(\overline{E}) \to \Lambda \otimes \tau^*(\overline{L^{-1}}) \to 0\]

\noindent This gives a complex antilinear map $\mathrm{Ext}^1(\Lambda \otimes L^{-1},L) \to \mathrm{Ext}^1(\Lambda\otimes \tau^*(\overline{L})^{-1},\tau^*(\overline{L}))$. We obtain the following result 

\begin{Proposition}
\upshape The bundle $\E$ admits a real structure $\tau_{\E}: \E \to \E$ over the real variety $J_0$ which induces an action $\tau: \bP(\E_0) \to \bP(\E_0)$ for which the fixed point set $\bP(\E)^{\tau}$ admits a fibration

\[\R\bP(\E^{\tau}) \to J^{\tau}_0\]
where the fibers are real projective spaces of dimension $g-1$.
\end{Proposition}

As a consequence of Lemma \ref{Degree-1-Lines}, we can describe the rational curves $f: \bP^1 \to M_{\Lambda}(2,1)$ in $M_A$ which have image preserved by $\tau$.

\begin{theorem}
\label{Deg-1-Real-Rational-Curves}
Let $f \in M_A$ be a rational curve $f: \bP^1 \to M_{\Lambda}(2,1)$ with image preserved by $\tau$. Then there is a real line bundle $L$ so that $f$ is a composition
\[\bP^1 \xrightarrow[]{h} \bP_L \to \bP(\E_0)\]
where $h$ has degree $1$ and $\bP_L$ is the fiber of $\bP(\E_0)$ over the isomorphism class $[L] \in J_0$ as above. Moreover, up to a change of coordinates, $f$ satisfies $\tau \circ f = f \circ c$, where $c: \bP^1 \to \bP^1$ is the complex conjugation on $\bP^1$. In particular, the image of $f$ intersects the fixed point locus $M_{\Lambda}(2,1)^{\tau}$ in a real line identified with $\R\bP^1$.

\end{theorem}

\begin{proof}
By Lemma $3.1$, there is some $L \in J_0$ such that $f$ factors through a degree $1$ map to $\bP_L$ and $\tau \circ f$ factors through a map to $\bP_{\tau^*(\overline{L})}$. By injectivity of the map $M_B \to M_A$, we must then have $L \cong \tau^*(\overline{L})$. On $\bP^1$, there are only two real structures, which are up to a change of coordinates identified with either complex conjugation $c$ or the antipodal map $a:  [z_1,z_2] \mapsto [-\overline{z}_2,\overline{z}_1]$. Thus we have $\tau \circ f = f \circ a$ or $\tau \circ f = f \circ c$, and since $g$ has degree $1$, the latter case must hold.

\end{proof}

From Theorem \ref{Deg-1-Real-Rational-Curves}, we see that a real bundle $E \in M_{\Lambda}(2,1)^{\tau}$ is contained in a real rational curve of degree $1$ if and only if $E$ contains a real line subbundle of degree $0$.\\

\noindent \textbf{Remark:} The action of $\tau$ on extension spaces can be understood more concretely with Dolbeaut cohomology. For an extension
\[0 \to L \to E \to L^{-1}\otimes \Lambda \to 0\]
of holomorphic bundles, one can choose a splitting of $E$ as a $C^{\infty}$ bundle, in which case the operator $\overline{\partial_E}$ of $E$ can be written as 
\[
\begin{bmatrix}
\overline{\partial}_L & B\\
0 & \overline{\partial}_{L^{-1}\otimes \Lambda}
\end{bmatrix}
\]
Where $B \in \Omega^{0,1}(L^2 \otimes \Lambda^{-1})$ is a Dolbeaut representative of the extension class. We can apply $\tau$ componentwise to this matrix, which sends $B$ to $\tau^*(\overline{B}) \in \Omega^{0,1}(\tau^*(\overline{L})^2 \otimes \Lambda^{-1})$. In local coordinates, writing $\overline{\partial}_L = \overline{\partial}+A$ with $A \in \Omega^{0,1}$, we have $\tau^*(\overline{\partial}_L) = \overline{\partial}+ \tau^*(\overline{A})$ as a connection on the bundle $\tau^*(\overline{L})$.
\section{Real Extension lines and subbundles in genus $2$}
\label{Section-Genus-2-Nonsingular}
Throughout this section we will assume that $(\Sigma,\tau)$ is a real hyperelliptic curve of genus $2$. In this case, work of Newstead \cite{newstead1968stable} gives an explicit description of the moduli space $M_{\Lambda}(2,1)$ and the map $i_{\Lambda}: \bP(\E_0) \to M_{\Lambda}(2,1)$. Note that in genus $2$, both $\bP(\E_0)$ and $M_{\Lambda}(2,1)$ have the same complex dimension $3$.

\begin{Proposition} (\cite{newstead1968stable})
\label{Newstead-Description}\upshape 
Let $\Sigma$ be a Riemann surface of genus $2$, and let $\lambda_1,\dots,\lambda_6 \in \bP^1$ be the coordinates of the Weierstrass points of $\Sigma$ in an affine coordinate system $\bP^1 = \C \cup \{\infty\}$. The moduli space $M_{\Lambda}(2,1)$ is isomorphic to the intersection in $\bP^5$ of the two quadrics
\[x_1^2+\dots+x_6^2, \lambda_1 x_1^2 + \dots + \lambda_6 x_6^2\]
The map $i_{\Lambda}: \bP(\E_0) \to M_{\Lambda}(2,1)$ is surjective and generically finite of degree $4$. In particular, a generic element of $M_{\Lambda}(2,1)$ is contained in $4$ rational curves of degree $1$ by Lemma \ref{Degree-1-Lines}.
\end{Proposition}

Later work of Desale-Ramanan \cite{desale1976classification} in higher genus shows that the moduli space $M_{\Lambda}(2,1)$ can be recovered from the intersection of two quadrics for all hyperelliptic curves. In \cite{baird2020hyperelliptic}, it is proven that one can recover the real moduli space $M_{\Lambda}(2,1)^{\tau}$ as an intersection of real quadrics

\begin{Proposition}(\cite{baird2020hyperelliptic})
\label{Real-Hyperelliptic-Genus2-Types}
\upshape 
Let $(\Sigma,\tau)$ be a real hyperelliptic curve of genus $2$ and $(\Lambda,\tau_{\Lambda})$ a real line bundle on $(\Sigma,\tau)$. The moduli space $M_{\Lambda}(2,1)^{\tau}$ is isomorphic to the intersection of two real quadrics in $\R\bP^5$. For a real hyperelliptic curve with $2m = |W_0|$ real Weierstrass points, such that $\Lambda$ has $k$ odd circles, the diffeomorphism type of $M_{\Lambda}(2,1)^{\tau}$ is as follows

\begin{center}
\begin{tabular}{|c|c|c|}
\hline
$(\Sigma,\tau)$ type & $(m,k)$ & Diffeomorphism type \\
\hline 
$(1,0)$ & $(0,1)$ & $L(4,1)$\\
\hline 
$(1,1)$ & $(1,1)$ & $S^1 \times S^2$\\
\hline 
$(2,1)$ & $(2,1)$ & $\#_2 (S^1 \times S^2)$\\
\hline 
$(3.0)$ & $(3,1)$ & $\#_3 (S^1 \times S^2)$\\
\hline
$(3,0)$ & $(3,3)$ & $T^3$\\
\hline
\end{tabular}
\end{center}

In the leftmost column we include the topological type of the real curve $(\Sigma,\tau)$. These cover all real hyperelliptic types in the figures in section \ref{Genus-2-Figures}. The curves of type $(3,0)$ split into two cases depending on the topological type of the determinant $\Lambda$.
\end{Proposition}

Taking real parts, one obtains a map $i^{\tau}_{\Lambda}: \bP(\E_0)^{\tau} \to M_{\Lambda}(2,1)^{\tau}$. The fixed point set $J^{\tau}_0$ consists of $1$ component when $(\Sigma,\tau)$ has type $(1,0),(1,1)$, and $2,4$ components in the cases $(2,1),(3,0)$ respectively. All components of $J^{\tau}_0$ are $2$-tori. We claim that each component of  $\bP(\E_0)^{\tau}$ is orientable. Recall the identification
\[\bP(\E_0)^{\tau} \cong \R\bP(\E^{\tau}_0) \to J^{\tau}_0\]
Thus each component of $\bP(\E_0)^{\tau}$ is an $\R\bP^1$ bundle over a $2$-torus. The extension bundle $\E_0 = R^1p_*(\mathcal{L}_0^2 \otimes \Lambda^{-1})$ is the pullback of the bundle $\mathcal{F} = R^1p_*(\mathcal{L}_{-1})$ on $J_{-1}$ under the isogeny
\[J_0 \to J_{-1}, L\mapsto L^2 \otimes \Lambda^{-1}\]
which has degree $2$. The bundle $\mathcal{F}$ admits a real structure so that $\mathcal{F}^{\tau}$ pulls back to $\E^{\tau}_0$. We conclude that $\E^{\tau}_0$, and hence $\bP(\E_0)^{\tau}$, is orientable.

\subsection{Atiyah's work} We now set up the tools to prove our main theorem \ref{Genus-2-Thm-Intro}. We will study the map $i^{\tau}_{\Lambda}$ using classical work of Atiyah (\cite{atiyah1955complex}), as exposed by Donaldson in (\cite{donaldson2021atiyah}, Section 1.2).

\begin{definition}
Two stable bundles $E,E' \in M(2,1)$ of rank $2$ and degree $1$ are \textit{projectively isomorphic} if their projectivized bundles $\bP(E),\bP(E')$ are isomorphic. Equivalently, there is some holomorphic line bundle $L$ for which $E' \cong E \otimes L$. We let $M_P$ be the moduli space of projective isomorphism classes of elements of $M(2,1)$.
\end{definition}

Any element of $M_P$ can be written as $\bP(E)$, where $E$ is an extension 
\[0\to \Ox \to E \to L_1\to 0\]
 with $L_1$ a line bundle of degree $1$. Such an extension is determined by an extension class in $H^1(L^*_1) \cong H^0(L_1 \otimes K)^*$ by Serre duality, where we write $K$ for the canonical bundle of $\Sigma$. By Riemann-Roch, the space $H^0(L_1 \otimes K)$ has dimension $2$, so there is a canonical identification of the one dimensional projective spaces $\bP(H^1(L^*_1))\cong \bP(H^0(L_1 \otimes K))$. In particular, we can identify an extension class with an effective divisor of degree $3$ on $\Sigma$, thus we obtain a surjective map $i_P: \mathrm{Sym}^3(\Sigma) \to M_P$.

 \begin{lemma} (\cite{atiyah1955complex})
 \label{Atiyah-Theorem}
 \upshape 
  For a divisor $A+B+C \in \mathrm{Sym}^3(\Sigma)$, the divisors defining projectively equivalent bundles are precisely 
 \[A+\iota(B)+\iota(C), \iota(A) + B + \iota(C), \iota(A)+\iota(B)+C\]
 where $\iota: \Sigma \to \Sigma$ is the hyperelliptic involution. In particular, the map $i_P: \mathrm{Sym}^3(\Sigma)\to M_P$ has degree $4$.
 \end{lemma}

One has a commutative diagram 

\[
\begin{tikzcd}
\bP(\E_0) \arrow{d}{i_S} \arrow{r}{i_{\Lambda}} & M_{\Lambda}(2,1) \arrow{d}{\bP}\\
\mathrm{Sym}^3(\Sigma) \arrow{r}{i_P} & M_P
\end{tikzcd}
\]

The map $i_S: \bP(\E_0) \to \mathrm{Sym}^3(\Sigma)$ is given by sending a projective extension class in $\bP(\mathrm{Ext}^1(L^{-1}\otimes \Lambda,L))$ to the associated element in the linear system $\bP(H^0(L^{-2}\otimes \Lambda \otimes K))\cong \bP(H^1(L^2\otimes \Lambda^{-1}))$, while the map $\bP$ is the projectivization map $E \mapsto \bP(E)$.  All maps in the diagram are surjective. The maps $i_S,\bP$ have degree $16$, corresponding to the $16$ square roots of the trivial bundle $\Ox \in J_0$.\\

Consider now the addition of the real structure $\tau$. The spaces $\mathrm{Sym}^3(\Sigma), M_P$ admit actions of $\tau$ by $A+B+C \mapsto \tau(A) + \tau(B)+\tau(C)$ and $\bP(E) \mapsto \bP(\tau^*(\overline{E}))$, and all maps in the diagram are equivariant with respect to the actions of $\tau$ on $\bP(\E_0),M_{\Lambda}(2,1)$ defined previously. For instance, equivariance through the Serre duality identification follows from the fact that $\tau$ preserves the integration pairing  
\[H^0(L^{-2}\otimes \Lambda \otimes K) \times H^1(L^2 \otimes \Lambda^{-1}) \to \C\]
up to sign. Taking fixed points we obtain a diagram 

\begin{equation}
\label{Atiyah-Diagram}
\begin{tikzcd}
\bP(\E_0)^{\tau} \arrow{d}{i^{\tau}_S} \arrow{r}{i^{\tau}_{\Lambda}} & M^{\tau}_{\Lambda}(2,1) \arrow{d}{\bP^{\tau}}\\
\mathrm{Sym}^{3}(\Sigma)^{\tau} \arrow{r}{i^{\tau}_P} & M^{\tau}_P
\end{tikzcd}
\end{equation}

These maps are all finite maps between closed $3$-manifolds, which depend in an interesting way on the real structure. Using Atiyah's result, we can describe the divisor classes corresponding to bundles in $M^{\tau}_P$. 

\begin{lemma} 
\label{Projectively-Real-Divisors}
A bundle $E \in M_{\Lambda}(2,1)$ is projectively real, in the sense that $\tau^*(\overline{E}) \cong E \otimes L_0$ for a square root $L_0$ of the trivial bundle $\Ox$, if and only if it can be written as an extension with associated divisor $A+B+C \in \mathrm{Sym}^3(\Sigma)$ so that either $\tau(A)+\tau(B)+\tau(C) = A+B+C$ or, without loss of generality, $\tau$ fixes $A$ and $\tau(B)+\tau(C) = \iota(B)+\iota(C)$.

\end{lemma}

\begin{proof} This follows from Lemma \ref{Atiyah-Theorem}, since $\tau^*(\overline{E})$ is represented by the extension class $\tau(A) + \tau(B) + \tau(C)$.

\end{proof}

We will use heavily the following fact about bundles preserved by tensoring by a square root of the trivial line bundle

\begin{lemma} (\cite{garcia2018involutions})
\label{Fix-Point-Codim-Lemma}
Let $L_0$ be a nontrivial line bundle on $\Sigma$ so that $L_0^2 \cong \Ox_{\Sigma}$. The fixed point set of the automorphism $\phi_L: M_{\Lambda}(2,1) \to M_{\Lambda}(2,1), E \mapsto E \otimes L_0$ is isomorphic to the Prym variety of the \'Etale double Cover 

\[X_{L_0} = \mathrm{Spec}(\Ox \oplus L_0) \to \Sigma\]

In the case when $g(\Sigma) = 2$, we compute $g(X_{L_0}) = 3$ by the Riemann-Hurwitz formula, so that the Prym-Variety of the covering has dimension $1$. Thus outside of a codimension $2$ subset of $M_{\Lambda}(2,1)$, no element is fixed by tensoring by a nontrivial square root of the trivial bundle.

\end{lemma}

\begin{remark}
It is an important consequence of Lemma \ref{Fix-Point-Codim-Lemma}, that the set of real bundles which are preserved by tensoring by some $L_0$ form a union of circles. Thus throughout a generic real bundle $E \in M_{\Lambda}(2,1)^{\tau}$ refers to one outside a union of circles.
\end{remark}

As an application we prove that real divisors characterize precisely the extensions by real line bundles.

\begin{lemma}
\label{Real-Divisor-Equal-Real-Extension}
Suppose $E \in M_{\Lambda}(2,1)^{\tau}$ is a real bundle which is generic in the sense that $E \otimes L_0 \ncong E$ for all square roots $L_0$ of $\Ox_{\Sigma}$. Then for $L \in J_0$ and an extension 
\[0 \to L \to E \to L^{-1} \otimes \Lambda \to 0\]
the bundle $L$ is real if and only if the corresponding extension class in $\mathrm{Sym}^3(\Sigma)$ is real.

\begin{proof}
If $L$ is real, then the bundle $E$ must correspond to a real extension class since it is itself real. Conversely, if the extension class is real, then the bundle $E \otimes L^{-1}$ corresponds to a real extension class 
\[0 \to \Ox \to E \otimes L^{-1} \to L^{-2} \otimes \Lambda \to 0\]
in $H^1(L^{-2} \otimes \Lambda) \cong H^0(L^2 \otimes \Lambda^{-1} \otimes K)$. Thus $L^2 \otimes \Lambda^{-1} \otimes K$ has a real divisor, and since $K$ admits a real structure we see that $L^2$ is real. Moreover, the bundle $E \otimes L^{-1}$ must be real, so
\[E  \otimes L^{-1} \cong E \otimes \tau^*(\overline{L})^{-1}\]
which by the generic assumption on $E$ implies that $L$ is real.
\end{proof}

By Lemma \ref{Fix-Point-Codim-Lemma}, we note that the size of a generic fiber of the maps $i^{\tau}_S, \bP^{\tau}$ is equal to the number of real square roots of the trivial bundle $\Ox$. We see also that all elements of a generic fiber of $i_S$ map to distinct elements in $M_{\Lambda}(2,1)$. Indeed for an extension

\[0 \to L \to E \to L^{-1} \otimes \Lambda \to 0\]
the extensions in the same fiber of $i_S$ are all obtained by tensoring this extension by a square root $L_0$ of $\Ox$, which for a generic $E$ will yield different stable bundles. In particular, for a generic $E \in M_{\Lambda}(2,1)$ the four extension lines containing it correspond to the four distinct projective isomorphism classes in Lemma \ref{Atiyah-Theorem}. Thus for a generic real $E \in M_{\Lambda}(2,1)^{\tau}$, the four degree $0$ real line subbundles correspond exactly to the real divisors defining a bundle projectively equivalent to $E$, which is a set of size $0,2$ or $4$.\\

Recall that the line bundle $K$ admits a holomorphic real structure with no odd circles. A real divisor class $A+B+C \in \mathrm{Sym}^3(\Sigma)^{\tau}$ representing a real line bundle $L_3$ of degree $3$ is in the image of $i^{\tau}_S$ if and only if there is real square root of $L_3 \otimes K^{-1} \otimes \Lambda$. These are precisely the line bundles $L_3$ which have the same odd circles as $\Lambda$, i.e an odd number of the points $A,B,C$ lie in each odd circle of $\Lambda$. \\

We next note that one can determine the orientability of the components of $\mathrm{Sym}^{3}(\Sigma)^{\tau}$ using work of Okonek-Teleman. By identifying a divisor with an element of the linear system $H^0(L^{-1} \otimes K) \cong H^1(L)$ for $L \in J_{-1}$, one obtains an identification of $\mathrm{Sym}^3(\Sigma)$ with the projectivized bundle 

\[\bP(R^1p_*(\mathcal{L}_{-1})) \to J_{-1}\]

The line bundle $\mathrm{det}(R^1p_*(\mathcal{L}_{-1}))^{-1}$ is identified with the determinant line bundle $\delta^L$ over the Jacobian $J_{-1}$ considered by Okonek-Teleman in \cite{okonek2013abelian}. We conclude the following

\begin{Proposition}
\upshape \label{Okonek-Teleman-Lemma}
The bundle $R^1p_*(\mathcal{L}_{-1})$ admits a real structure, and there is an identification

\[\mathrm{Sym}^3(\Sigma)^{\tau} \cong \R\bP(R^1p_*(\mathcal{L}_{-1})^{\tau}) \to J^{\tau}_{-1}\]

If the set of fixed circles $\Sigma^{\tau}$ is separating in $\Sigma$, then $R^1p_*(\mathcal{L}_{-1})^{\tau}$ is orientable only over the components of $J_{-1}^{\tau}$ consisting of real line bundles $(L,\tau_L)$ for which $L^{\tau_L}$ is nonorientable over each component of $\Sigma^{\tau}$. If $\Sigma^{\tau}$ is nonseparating, then $R^1p_*(\mathcal{L}_{-1})^{\tau}$ is nonorientable over every component of $J^{\tau}_{-1}$.

\end{Proposition}

\begin{proof}
This follows from Theorem 4.15 in \cite{okonek2013abelian}. See also Lemma $3.4$ in \cite{baird2016cohomology}.
\end{proof}

\subsection{Real bundles on real hyperelliptic curves of genus $2$.}
\label{Case-by-Case-Subbundles}
We will next analyze the map $i^{\tau}_{\Lambda}: \bP(\E_0)^{\tau} \to M_{\Lambda}(2,1)^{\tau}$ for each of the five possibilities of topological type for $(\Sigma,\tau)$ and number of odd circles for $\Lambda$. The possibilities are all labelled in Proposition \ref{Real-Hyperelliptic-Genus2-Types}. Up to homeomorphism, the action of the real structure $\tau$ and hyperelliptic involution $\iota$ can be identified with one of the figures in subsection \ref{Genus-2-Figures}. We will refer to these figures frequently for geometric arguments involving divisors.\\

On a hyperelliptic curve, the $16$ two-torsion points on the Jacobian are of the form $\Ox(w_i-w_j)$ with $w_i,w_j$ Weierstrass points. The real two torsion points are then such bundles for which $w_i,w_j$ are either real or satisfy $\tau(w_i) = w_j$. We can thus identify the size of a generic fiber of $i^{\tau}_S,\bP^{\tau}$.\\

\noindent \textbf{Curve of type $(1,0)$ with one odd circle:} 
\label{Bundles-(1,0)}

\begin{theorem} Suppose $(\Sigma,\tau)$ is a real hyperelliptic curve of genus $2$ of type $(1,0)$, and $(\Lambda,\tau_{\Lambda})$ is a real line bundle on $\Sigma$ of degree $1$ with one odd circle. Then every real bundle $E \in M_{\Lambda}(2,1)^{\tau}$ contains a real line subbundle of degree $0$. In particular, the map  
\[i^{\tau}_{\Lambda}: \bP(\E_0)^{\tau} \to M_{\Lambda}(2,1)^{\tau}
\]

\noindent is surjective. There are nonempty open subsets of $M_{\Lambda}(2,1)^{\tau}$ with dense union of real bundles containing exactly $2$ and $4$ real line subbundles.

\end{theorem}

\begin{proof} 

By lemma \ref{Projectively-Real-Divisors}, any real bundle $E$ is represented by a divisor $A+B+C \in \mathrm{Sym}^3(\Sigma)$ which is either preserved by $\tau$ or $\tau$ fixes $A$ and $\tau(B)+\tau(C) = \iota(B)+\iota(C)$. Assume that $A+B+C$ is not real. Thus either $\iota(\tau(B))= C, \iota(\tau(C)) = C$ or $C = \iota(\tau(B))$. By inspection of Figure \ref{genus-2-type-(1,0)}, the first case does not occur. In the second case, the divisor $\iota(A)+\iota(B)+C$ is real and projectively equivalent to $A+B+C$ since $\tau,\iota$ commute, so any generic real bundle in the corresponding extension class is represented by a real extension class. By Lemma \ref{Real-Divisor-Equal-Real-Extension}, this implies that the image of $i^{\tau}_{\Lambda}$ contains an open dense subset, so by compactness of $\bP(\E_0)^{\tau}$ it is all of $M_{\Lambda}(2,1)^{\tau}$.
\newline 
\indent One can produce bundles containing two real line subbundles by considering divisors $A+B+C$ with $A \in \Sigma^{\tau}$ and $C = \tau(B)$ with $B,C \notin \Sigma^{\tau}$. These are real, and only the divisor $A+\iota(B)+\iota(C)$ is real and defines a projectively equivalent bundle to $A+B+C$. One can produce bundles with four real line subbundles by taking $A,B,C \in \Sigma^{\tau}$. Since $\iota$ preserves the circle $\Sigma^{\tau}$, all divisors defining projectively equivalent bundles are real in this case. Both sets have dimension $3$ and hence define open subsets.

\end{proof}

\end{lemma}

For a type $(1,0)$ curve, there is only one component of $J_{d}^{\tau}$ in each degree $d$. There is thus only one component of $\mathrm{Sym}^3(\Sigma)^{\tau}$, which is orientable by Proposition \ref{Okonek-Teleman-Lemma}. There are $3$ conjugate pairs of Weierstrass points, hence $4$ real roots of $\Ox$, and the maps $i^{\tau}_S, \bP^{\tau}$ have generic fibers of size $4$. The moduli space $M_{\Lambda}(2,1)^{\tau}$ is diffeomorphic to the lens space $L(4,1)$. Using the fact that $L(4,1)$ is the unit tangent bundle of $\R \bP^2$, one can produce a map of degree $4$
\[S^1 \times T^2 \to L(4,1)\]

by combining the elliptic quotient $T^2 \to S^2$ and the covering $S^2 \to \R\bP^2$, and then taking the differential on unit tangent bundles. This cannot be the map $i_{\Lambda}^{\tau}: \bP(\E_0)^{\tau} \to M_{\Lambda}(2,1)^{\tau}
$, as it has generic fibers of size $4$ but is potentially related.\\

\noindent \textbf{Curve of type $(3,0)$ with three odd circles:}

\begin{theorem}
Suppose $(\Sigma,\tau)$ is a real hyperelliptic curve of genus $2$ of type $(3,0)$, and $(\Lambda,\tau_{\Lambda})$ is a real line bundle on $\Sigma$ of degree $1$ with three odd circles. The map 

\[i^{\tau}_{\Lambda}: \bP(\E_0)^{\tau} \to M_{\Lambda}(2,1)^{\tau}
\]

\noindent is surjective. In particular, a real bundle $E \in M_{\Lambda}(2,1)$ contains a real line subbundle of degree $0$. Moreover, generically there are four such subbundles. 
\end{theorem}

\begin{proof}
The real projective bundle $\bP(\E_0)^{\tau}$ has $4$ components, and the image of $i^{\tau}_S$ consists of the component $T \subset \mathrm{Sym}^3(\Sigma)^{\tau}$ diffeomorphic to $T^3$ of divisors $A+B+C$ with $A,B,C$ in distinct components of $\Sigma^{\tau}$. Thus a generic bundle in the image of $i^{\tau}_{\Lambda}$ is in a fiber of $i_{\Lambda}^{\tau}$ of size four, i.e contains four real line subbundles of degree $0$.\\

Suppose that the map $i_{\Lambda}^{\tau}$ is not surjective. Since every element of $M_{\Lambda}(2,1)^{\tau}$ corresponds to a divisor which is projectively real in the sense of Lemma \ref{Projectively-Real-Divisors}, by counting dimensions, there must be a two dimensional subset of the component $T$ consisting of divisors $A+B+C$ which satisfy $\tau(A) = A, \tau(B)+\tau(C) = \iota(B)+\iota(C)$. But for $A,B,C \in \Sigma^{\tau}$ in distinct fixed circles, this condition implies that $B,C$ must be Weierstrass points, hence one can only obtain a $1$ dimensional subset of $T$. 

\end{proof}

 There are four components of $\bP(\E_0)^{\tau}$, which are orientable $S^1$ bundles over $T^2$. Since all six Weierstrass points of $\Sigma$ are real, the fibers of $i^{\tau}_S$ have size $16$, and consist of $4$ points in each component of $\bP(\E_0)^{\tau}$. Since $M_{\Lambda}(2,1)^{\tau}$ is diffeomorphic to $T^3$ by Proposition \ref{Real-Hyperelliptic-Genus2-Types}, one might hope that the map $i^{\tau}_{\Lambda}$ is identified with the trivial cover $\sqcup_4 T^3 \to T^3$. We prove this in section \ref{Lange-Narasimhan-Applications}. In particular, no bundles in the branch locus of the map $i_{\Lambda}$ are real in this case.\\

\noindent \textbf{Curves of type $(1,1),(2,1),(3,0)$ with one odd circle:}

\begin{theorem}
\label{Non-Surjective-Types-Theorem}
Suppose $(\Sigma,\tau)$ is a real hyperelliptic curve of genus $2$ of one of the types  $(1,1),(2,1),(3,0)$, and $(\Lambda,\tau_{\Lambda})$ is a real line bundle on $\Sigma$ of degree $1$ with one odd circle. The map 

\[i^{\tau}_{\Lambda}: \bP(\E_0)^{\tau} \to M_{\Lambda}(2,1)^{\tau}
\]

\noindent is not surjective. In particular, there are real bundles $E \in M_{\Lambda}(2,1)^{\tau}$ with no real line subbundles of degree $0$. Moreover, there are nonempty open sets of real bundles with dense union containing precisely $0,2,4$ real line subbundles of degree $0$.

\end{theorem}

\begin{proof}
We may treat all three cases simultaneously by focusing on a slice of the picture around a fixed circle as shown in Figure \ref{NoRealSubbundlesFamily}.
One can produce open sets containing exactly two and four real line subbundles by the same reasoning as for curves of type $(0,1)$. For two subbundles, we pick $A \in \Sigma^{\tau}$ and $B \notin \Sigma^{\tau}$ with $C = \tau(B)$. For four, we consider $A,B,C \in \Sigma^{\tau}$.\newline

Producing open sets of real bundles containing no degree $0$ real line subbundles requires more care. As in Figure \ref{NoRealSubbundlesFamily}, one can produce a three dimensional family of divisors defining projectively real bundles, such that a two dimensional subset consists of real divisors. Label by $U$ the corresponding open set in $M_{\Lambda}(2,1)$. \newline \indent By passing to an open subset of $U$, using Lemma \ref{Fix-Point-Codim-Lemma}, we may assume that all elements $E \in U$ satisfy $\tau(\overline{E}) \cong E \otimes L^{\R}_0$, where $L^{\R}_0$ is a real square root of the trivial bundle $\Ox$. The antiholomorphic involutions  $L^{\R}_{0} \otimes \tau$ on $M_{\Lambda}(2,1)$ all have fixed subsets that are empty or real Lagrangian in $M_{\Lambda}(2,1)$, hence they intersect the fixed point set of $L^{\R}_0$ in a union of circles. In particular, the intersection of the fixed subsets $M_{\Lambda}(2,1)^{L^{\R}_{0}\otimes \tau}, M_{\Lambda}(2,1)^{(L^{\R}_{0})'\otimes \tau}$ for distinct $L^{\R}_{0},(L^{\R}_{0})'$ has dimension $1$. Since $U$ intersects $M_{\Lambda}(2,1)^{\tau}$ in a subset of dimension $2$, it then contains an open subset of $M_{\Lambda}(2,1)^{\tau}$ in the intersection $U \cap M_{\Lambda}(2,1)^{\tau}$.

\end{proof}

\begin{remark}
\label{Wobbly-Remark}
By a  continuity argument, the real bundles in $M_{\Lambda}(2,1)^{\tau}$ with different numbers of real line subbundles of degree $0$ are separated by closed subsets of $M_{\Lambda}(2,1)^{\tau}$ consisting of bundles with strictly less than four line subbundles of degree $0$. These are \textit{wobbly bundles}, as in \cite{Pal2017}. They form the real points of a hypersurface in $M_{\Lambda}(2,1)$ of degree $32$.
\end{remark}

\begin{figure}
\label{NoRealSubbundlesFamily}
\centering
\includegraphics[width=100mm,height = 100mm]{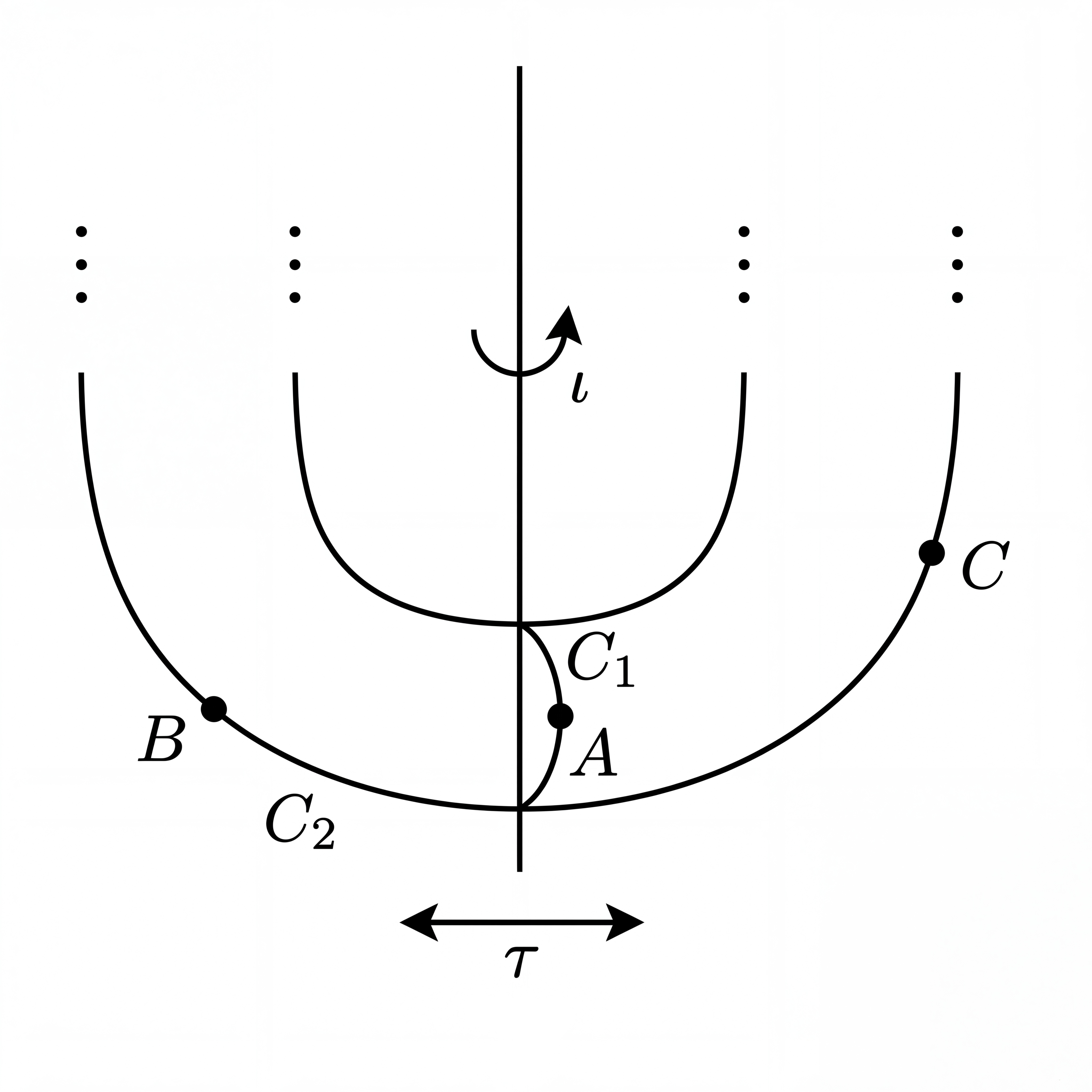}
\caption{The family used in the proof of Theorem \ref{Non-Surjective-Types-Theorem}. One chooses $A$ in the fixed circle $C_1$, and $B,C$ in the arc $C_2$ consisting of points on which $\tau(p) = \iota(p)$. By taking $B = \tau(C)$, we obtain a $2$ dimensional subfamily of real divisors.} 
\end{figure}
\newpage

\section{Applications of work of Lange-Narasimhan}
\label{Lange-Narasimhan-Applications}
We will work now with a real curve $(\Sigma,\tau)$ of arbitrary genus $g = g(\Sigma)$ which has fixed points, and a real point $p \in \Sigma^{\tau}$.  We recall the following result of Lange-Narasimhan \cite{lange1983maximal} (see also \cite{oxbury1998subvarieties}).

\begin{lemma}
\label{Lange-Narasimhan-Lemma}
Let $F$ be an extension 
\[0 \to n_0 \to F \to n_0^{-1} \otimes \mathrm{det} (F) \to 0 \]
where $n_0 \subset F$ is a maximal subbundle. There is a bijection, given by $\Ox(D) = n^{-1}n^{-1}_0 \otimes \mathrm{det}(F)$, between

\begin{enumerate}
\item Line subbundles $n \subset F$ distinct from $n_0$ and of maximal degree
\item Line bundles $\Ox(D)$ with degree $\mathrm{deg}(D) = \mathrm{deg}(F)-2\mathrm{deg}(n_0)$ and such that the extension class of $F$ lies in the linear span of the image of $D$ under the map 
\[\Sigma \xrightarrow[]{} \bP(H^0(Kn_0^{-2} \otimes \mathrm{det}(F))^*)\]
\end{enumerate}
\end{lemma}

\subsection{Topological types in genus $2$}
We return now to the case of a real curve $(\Sigma,\tau)$ of genus $2$, and a real line bundle $\Lambda$ of degree $1$. Lemma \ref{Lange-Narasimhan-Lemma} allows us to identify the topological types of real subbundles of a stable real bundle $E \in M_{\Lambda}(2,1)$ 

\begin{theorem}
\label{Topological-Types-Theorem}
Let $(\Sigma,\tau)$ be a real hyperelliptic curve of genus $2$, and $\Lambda$ a real line bundle on $\Sigma$ of degree $1$. The following hold for a stable real bundle $E \in M_{\Lambda}(2,1)$.

\begin{enumerate}
\item If $\Lambda$ has one odd circle, then there are at most two distinct topological isomorphism classes of real line subbundles of degree $0$ of $E$.
\item If $\Lambda$ has three odd circles, there are four distinct topological isomorphism classes of real line subbundles of degree $0$ of $E$.
\end{enumerate}
\end{theorem}

\begin{proof}
Assume $E$ has a real line subbundle $L$ of degree $0$. Applying Lemma \ref{Lange-Narasimhan-Lemma}, the other line subbundles of $E$ of degree $0$ are in bijection with the points in $\Sigma$ mapping to the class of $E$ under the degree $3$ map
\[|KL^{-2}\otimes \Lambda|: \Sigma \to \bP^1 \]
A point $p \in \Sigma$ mapping to $E$ corresponds to a subbundle $L \otimes \Lambda^{-1} \otimes \Ox(p)$, which is real if and only if $p \in \Sigma^{\tau}$. If $E$ is real, its extension class lies in the real locus $(\bP^1)^{\tau} \cong \R \bP^1$ and so the real line subbundles of $E$ distinct from $L$ are parametrized by real points in the real divisor obtained by pulling back $E \in \bP^1$ to $\Sigma$ under $|KL^{-2}\otimes \Lambda|$.\\

When $\Lambda$ has one odd circle, the real divisors of $KL^{-2}\otimes \Lambda$ have real points in atmost two components of $\Sigma^{\tau}$, wheres they have one real point in each component of $\Sigma^{\tau}$ when $\Lambda$ has three odd circles. This completes the proof by counting odd circles.
\end{proof}

\begin{corollary}
Suppose $(\Sigma,\tau)$ is a real hyperelliptic curve of genus $2$ of type $(3,0)$, and $(\Lambda,\tau_{\Lambda})$ is a real line bundle on $\Sigma$ of degree $1$ with three odd circles. The map 

\[i^{\tau}_{\Lambda}: \bP(\E_0)^{\tau} \to M_{\Lambda}(2,1)^{\tau}
\]

\noindent from the real extension bundle $\bP(\E_0)^{\tau}$ is surjective, and is identified with the trivial $4$ sheeted covering map of the $3$-torus
\[\sqcup_4 T^3 \to T^3\]
\end{corollary}

\begin{proof}
As in section \ref{Case-by-Case-Subbundles}, there are four components of $\bP(\E_0)^{\tau}$. By Lemma \ref{Topological-Types-Theorem}, the map $i_{\Lambda}^{\tau}$ has degree $1$ on each component, so is a homeomorphism on each component. We conclude by recalling that $M_{\Lambda}(2,1)^{\tau}\cong T^3$ by Proposition \ref{Real-Hyperelliptic-Genus2-Types}. 
\end{proof}

\subsection{Higher genus results}

 \begin{theorem}
 Let $(\Sigma,\tau)$ be a real curve of genus $g\geq3$ and $\Lambda,\beta$ real line bundles of degree $1,0$ on $\Sigma$. The following hold 
 \begin{enumerate}
 \item A real bundle $E \in M_{\Lambda}(2,1)$ which contains a maximal line subbundle $n \subset F$ of degree $0$ in fact contains a real line subbundle of maximal degree. Generically, $E$ has the unique maximal line subbundle $n$ of degree $0$ and thus $n$ is real.
 \item For $g>3$, the same statement is true for line subbundles $n \subset F$ of degree $-1$ of a real bundle $F \in M_{\Lambda}(2,0)$.
\end{enumerate}
\end{theorem}

\begin{proof} We consider the extension space  $\bP(H^0(Kn^{-2} \otimes \mathrm{det}(E))^*)$,  $\bP(H^0(Kn^{-2} \otimes \mathrm{det}(F))^*)$ in both cases. When $n$ has degree $0$, the projectivized extension space has dimension $g-1$, and maximal line subbundles of $E$ distinct from $n$ are determined by the points in $\Sigma$ mapping to the extension class of $E$ under the map from the linear system $|Kn^{-2}\otimes \mathrm{det}(E)|$. When $g\geq 3$, the image of $\Sigma$ has positive codimension, so generically $n$ is the only subbundle of $E$ of degree $0$. \\

Let $M \subset M_{\Lambda}(2,1)$ be the subvariety of bundles containing a maximal degree $0$ line subbundle, then $\tau$ preserves $M$ and the image $\bP(\E_{0})^{\tau} \to M^{\tau}$ of the projectived real extension bundle defined in Section \ref{Real-Extension-Lines-Section} has dense image. It is thus all of $M^{\tau}$ by compactness.\\

When $n$ has degree $-1$, the projectivized extension space $\bP(H^0(Kn^{-2} \otimes \mathrm{det}(F))^*)$ has dimension $g$, and line subbundles of degree $-1$ distinct from $n$ are parametrized by divisors $p_1+p_2$, such that the class of $F$ lies in the linear span of $p_1+p_2$ under the map. For $g > 3$, the secant variety of $\Sigma$ in $\bP(H^0(Kn^{-2} \otimes \mathrm{det}(F))^*)$ has positive codimension, so we obtain the desired result.

\end{proof}

It is in fact true in any genus $g\geq 2$ that the maximal number of subbundles of degree $0$ of an element $E \in M_{\Lambda}(2,1)$ is four, and three on non-hyperelliptic curves (see \cite{oxbury1998subvarieties}, Corollary $2.2$). One can adapt the construction of such bundles in \cite{oxbury1998subvarieties} to the real setting, to produce bundles where all three or four maximal line subbundles are real. On a genus $3$ curve, a bundle $F \in M_{\beta}(2,0)$ generically has $8$ line subbundles of degree $-1$ (see \cite{oxbury1998subvarieties}). We will prove that it is possible for no real line subbundles to exist in this case in future work.

\bibliography{references}
\bibliographystyle{alpha}

\end{document}